\documentclass[11pt,letterpaper,reqno]{amsart}
\usepackage{tikz}
\usetikzlibrary{positioning, shapes.geometric, arrows.meta}
\usepackage{amssymb}
\usepackage{amsmath}
\usepackage{amsthm}
\usepackage{amsfonts}
\usepackage{bbm}

\usepackage{pgfplots}
\pgfplotsset{compat=1.18} 

\usepackage{graphicx}
\usepackage[T1]{fontenc}
\usepackage{doi}
\usepackage{float} 
\usepackage{bookmark}
\usepackage{hyperref}
\hypersetup{pdfstartview={FitH}}
\newcommand{\C}{\mathbb{C}}

\newcommand{\norm}[1]{\lVert #1 \rVert}
\newcommand{\abs}[1]{| #1 |}
\newcommand{\bv}{\mathbf{v}}
\newcommand{\bx}{\mathbf{x}}
\newcommand{\bw}{\mathbf{w}}
\newcommand{\FF}{\mathbb{F}}

\newcommand{\tr}{\text{tr}}

\newtheorem{thm}{Theorem}[section]
\newtheorem{lem}[thm]{Lemma}
\newtheorem{prop}[thm]{Proposition}

\newtheorem{claim}{Claim}
\newtheorem{ques}[thm]{Question}
\newtheorem{conj}[thm]{Conjecture}
\theoremstyle{definition}

\newtheorem{remark}[thm]{Remark}
\newtheorem{defn}[thm]{Definition}
\numberwithin{equation}{section}

\makeatother

\begin{document}

\title{Inertia, Independence and Expanders}

% Authorship order changed per Clive's request
\author[Quanyu Tang]{Quanyu Tang}
\author[Shengtong Zhang]{Shengtong Zhang}
\author[Clive Elphick]{Clive Elphick}
\address{School of Mathematics and Statistics, Xi'an Jiaotong University, Xi'an 710049, P. R. China}
\email{tang\_quanyu@163.com}

\address{Department of Mathematics, Stanford University, Stanford, CA 94305, USA}
\email{stzh1555@stanford.edu}

\address{School of Mathematics, University of Birmingham, Birmingham, UK}
\email{clive.elphick@gmail.com}

\subjclass[2020]{Primary 05C50; Secondary 05C35, 05C48, 05C69.}
% 05C50      Graphs and linear algebra (matrices, eigenvalues, etc.)
% 05C35  	 Extremal problems in graph theory
% 05C48      Expander graphs
% 05C69      Vertex subsets with special properties (dominating sets, independent sets, cliques, etc.)

% \keywords{xxx}

\begin{abstract}
Let $G$ be a graph on $n$ vertices, independence number $\alpha(G)$, Lov{\'a}sz theta function $\vartheta(G)$, and Shannon capacity $\Theta(G)$. We define $n_{\ge0}(G)$ to be the minimum number of non-negative eigenvalues taken over all Hermitian weighted adjacency matrices of $G$. It is well known that $\alpha(G) \le \Theta(G) \le \vartheta(G)$ and $\alpha(G) \le n_{\ge0}(G)$.

Continuing a long line of work, we investigate the relationships between $ \alpha(G) $, $ \vartheta(G) $, $\Theta(G)$, and $ n_{\ge 0}(G) $. We prove a conjecture of Kwan and Wigderson, showing that for every integer $k$, there exists a graph $G$ with $\alpha(G) \leq 2$ and $n_{\ge 0}(G) \ge k$. In addition, we prove that for every integer $k$, there exists a graph $G$ with $\Theta(G) \leq 3$ and $n_{\ge 0}(G) \ge k$. Both results rely on a new observation: if the complement of $G$ contains a good spectral expander, then $n_{\geq 0}(G)$ must be large. We also show that $\vartheta(G)$ can be exponentially larger than $n_{\ge 0}(G)$, improving a recent result of Ihringer.
\end{abstract}
\maketitle

% Continuing a long line of work, we investigate the relationships between $ \alpha(G) $, $ \vartheta(G) $, $\Theta(G)$ and $ n_{\ge 0}(G) $. We prove a conjecture of Kwan and Wigderson, showing that for every integer $ k $, there exists a graph $ G $ with $ \alpha(G) \leq 2 $ and $ n_{\ge 0}(G) \ge k$. Our proof relies on a new observation: if the complement of $G$ contains a good spectral expander, then $n_{\geq 0}(G)$ must be large. We also show that $ \vartheta(G) $ can be exponentially larger than $ n_{\ge 0}(G) $, improving a recent result of Ihringer. 

\section{Introduction}
The inertia is an important tool in spectral graph theory that has recently attracted significant attention. We begin by recalling its definition.

Let $G = (V, E)$ be a simple graph with vertex set $V$ and edge set $E$. A \emph{weighted adjacency matrix} of $G$ is a Hermitian matrix $A$ with rows and columns indexed by $V$, such that $A_{ij} \neq 0$ only if $ij$ is an edge in $G$. The \emph{unweighted adjacency matrix} of $G$ is the matrix $A_G$ with $(A_{G})_{ij} = 1$ if $ij$ is an edge in $G$ and $(A_{G})_{ij} = 0$ otherwise.

For a Hermitian \(n \times n\) matrix \(X\), the \emph{inertia} \(n_{\geq 0}(X)\) is defined as the number of nonnegative eigenvalues of \(X\). Given a graph \( G \), we define its \emph{(weighted) inertia} as the minimum value of \( n_{\geq 0}(A) \) taken over all weighted adjacency matrices \( A \) of \( G \); we denote this quantity by \( n_{\geq 0}(G) \).\footnote{Throughout this paper, unless otherwise specified, when we refer to ``the inertia bound,'' we mean the minimum of \( n_{\geq 0}(A) \) over all weighted adjacency matrices \( A \) of \( G \), rather than the inertia of its unweighted adjacency matrix. We use \( n_{\geq 0}(A_G) \) to denote the number of nonnegative eigenvalues of the unweighted adjacency matrix \( A_G \).}

We denote the independence number of a graph \( G \) by \( \alpha(G) \). The celebrated inertia bound due to Cvetkovi\'{c} \cite{Cvetkovic1971} and Calderbank--Frankl \cite{Calderbank1992} states:
\begin{equation}\label{eq:inertia-bound-1-1}
\alpha(G) \leq n_{\geq 0}(G).\end{equation}
This bound, with different choices of the weighted adjacency matrix $A$, has a wide range of interesting applications in extremal set theory. For example, Calderbank--Frankl \cite{Calderbank1992} used the inertia bound to prove the Erd\H{o}s--Ko--Rado theorem. We refer to the introductory section of \cite{Kwan2024} for a comprehensive survey on this topic.

Given the extensive applications of the inertia bound, it is natural to ask whether it is tight. In 2004, Chris Godsil, in unpublished lecture notes, asked whether \( \alpha(G) = n_{\ge 0}(G) \) holds for all graphs, given an appropriately chosen weighted adjacency matrix. The first progress on this question appeared in 2018, when Sinkovic~\cite{Sinkovic} showed that for the Paley graph on 17 vertices, \( \alpha(G) < n_{\ge 0}(A) \) for all real symmetric weighted adjacency matrices $A$ of $G$. 
Further developments on Godsil’s question were made in~\cite{Mancinska},~\cite{Wocjan2022}, and~\cite{Kwan2024}, 
culminating in the following conjecture of Kwan and Wigderson.
\begin{conj}[\cite{Kwan2024}, Conjecture~4.2]
For every integer \(k\), there exists a graph \(G\) with \(\alpha(G) = 2\) and \(n_{\geq 0}(A) \geq k\) for every weighted adjacency matrix \(A\) of \(G\).
\end{conj}
In this paper, we confirm this conjecture by proving the following theorem. 
\begin{thm}
\label{thm:main}
For every positive integer $n$, there exists a graph $G_n$ on $n$ vertices such that $\alpha(G_n) \leq 2$ and $n_{\geq 0}(G_n) = \Omega(n^{1/9})$.
\end{thm}
Thus, we show that $n_{\geq 0}(G)$ cannot be upper bounded by any function of $\alpha(G)$, giving a strongly negative answer to Godsil's question.

% Our result parallels a result of Alon \cite{Alon2019}, who showed that the Lov\'{a}sz theta function is not upper bounded by any function of the Shannon capacity.

As we are unaware of any existing bound that can be used to prove Theorem~\ref{thm:main}, we introduce a new lower bound for \( n_{\geq 0}(G) \). We recall the classical definition of spectral expanders.
\begin{defn}
An \emph{\((n, d, \lambda)\)-graph} is a \(d\)-regular graph \(G\) on \(n\) vertices such that
\[
\max\{ |\lambda_2(G)|, |\lambda_n(G)| \} \leq \lambda,
\]
where \(\lambda_1(G) = d \geq \lambda_2(G) \geq \cdots \geq \lambda_n(G)\) are the eigenvalues of the unweighted adjacency matrix $A_G$ of \(G\).
\end{defn}

We refer to the survey of Hoory, Linial and Wigderson~\cite{Hoory2006} for a comprehensive introduction to expander graphs and their applications. Our next theorem shows that the inertia of $G$ is large when its complement graph $\overline{G}$ contains a good spectral expander.
\begin{thm}\label{thm:ndlambda1}
    Suppose $G = (V, E)$ is an $n$-vertex graph, such that its complement contains an $(n, d, \lambda)$-graph $\Gamma$ as a spanning subgraph with $d < \frac{n}{2} - 1$. Then we have
    $$n_{\geq 0}(G) \geq  \left(\frac{d + \lambda}{8\lambda}\right)^{1/3}.$$
\end{thm}
Combining this theorem with Alon's celebrated construction of triangle-free pseudorandom graphs~\cite{Alon1994}, we obtain Theorem~\ref{thm:main}. Our result parallels the classical ratio bound \cite{Lovasz79} for the Lov\'{a}sz theta function. A notable difference is that our proof relies on bounds for both the second and the least eigenvalue of the complement, while the ratio bound only needs a bound on the least eigenvalue. 

We also investigate how existing lower bounds on \( n_{\geq 0}(G) \) relate to Theorem~\ref{thm:main}. In 2019, Wocjan and Elphick~\cite[Theorem~6 and Remark~7]{Wocjan2019} established the following lower bound for the inertia in terms of variants of the chromatic number:
\begin{equation}
\label{eq:existing-bounds}
n_{\geq 0}(G) \geq \frac{|G|}{\xi_f(G)} \geq \frac{|G|}{\xi(G)},
\end{equation}
where \( |G| \), \( \xi(G) \), and \( \xi_f(G) \) denote the number of vertices, the orthogonal rank, and the projective rank of \( G \), respectively; see Appendix~\ref{sec:orthogonal-rank} for precise definitions of \( \xi(G) \) and \( \xi_f(G) \). By combining this lower bound with a classical construction of nearly orthogonal vector sets due to Alon and Szegedy~\cite{AlonSzegedy1999}, we obtain the following weaker version of Theorem~\ref{thm:main}, in which the constant 2 is replaced by a larger absolute constant.

% By combining the orthogonal rank bound and Alon--Szegedy's classical construction of almost orthogonal sets, we can establish the following weaker version of Theorem~\ref{thm:main} where the constant $2$ is replaced with some larger absolute constant.
\begin{prop}
    \label{prop:large-indep}
    There exists an absolute constant $C > 0$ such that for every positive integers $n \geq 2$ and $k \leq n$, there exists an $n$-vertex graph $G$ satisfying $$\alpha(G) \leq k \quad \text{ and }\quad  n_{\geq 0}(G) \geq n^{1 - \frac{C \log \log(k + 2)}{\log(k + 2)}}.$$
    In particular, for every $n$, there exists an $n$-vertex graph $G$ such that 
    $$\alpha(G) = O(1) \quad \text{ and }\quad  n_{\geq 0}(G) \geq n^{0.99}.$$
    There also exists an $n$-vertex graph $G$ such that 
    $$\alpha(G) = n^{o(1)} \quad \text{ and }\quad  n_{\geq 0}(G) \geq n^{1 - o(1)}.$$
\end{prop}

In the second paragraph of~\cite[p.~3200]{Kwan2024}, Kwan and Wigderson wrote: \begin{quote}
\emph{We also remark that if one is interested in the largest possible multiplicative gap between
the independence number and the inertia bound, ...}
\end{quote}and subsequently treated this as an open problem. Note that the final part of Proposition~\ref{prop:large-indep} shows that the multiplicative gap between the independence number and the inertia bound can be as large as \( n^{1 - o(1)} \). Therefore, up to an \( n^{o(1)} \) factor, our result resolves this open problem.

On the other hand, Rosenfeld \cite{Rosenfeld} showed that $\alpha(G) = 2$ implies $\xi(G) \geq \frac{|G|}{2}$. We now show that even the strongest bound in \eqref{eq:existing-bounds}, which involves the projective rank \( \xi_f(G) \), does not yield a nontrivial bound when $\alpha(G) = 2$.
\begin{prop}\label{prop:projective-independence-2}
Let \( G \) be a graph on \( n \) vertices with \( \alpha(G) = 2 \). Then \( \xi_f(G) \geq \frac{n}{2} \).
\end{prop}
% \begin{remark}
% Godsil and Sobchuk~\cite[Lemma~9.3]{Godsil2024} showed that \( \alpha(G) = 2 \) if and only if \( \alpha_q(G) = 2 \). Combined with~\cite[Corollary~6.13.2]{Mancinska}, it follows that if \( G \) is a vertex-transitive graph with \( n \) vertices and \( \alpha(G) = 2 \), then \( \xi_f(G) \leq \frac{n}{2} \). Hence, the inequality in Proposition~\ref{prop:projective-independence-2} is tight, as equality is achieved for vertex-transitive graphs. \end{remark}
\begin{remark}
The inequality in Proposition~\ref{prop:projective-independence-2} is tight. This follows from a result of Godsil and Sobchuk~\cite[Lemma~9.3]{Godsil2024}, who showed that \( \alpha(G) = 2 \) if and only if \( \alpha_q(G) = 2 \), where $\alpha_q(G)$ is the quantum independence number. Combined with~\cite[Corollary~6.13.2]{Mancinska}, it follows that if \( G \) is a vertex-transitive graph on \( n \) vertices with \( \alpha(G) = 2 \), then \( \xi_f(G) = \frac{n}{2} \).
\end{remark}

We remark, however, that since the Alon--Boppana bound (\cite{Al}, \cite{Ni}, \cite{Ni1}) implies \( \lambda = \Omega(d^{1/2}) \), Theorem~\ref{thm:ndlambda1} cannot yield a bound better than \( n_{\geq 0}(G) = \Omega(n^{1/6}) \) for any graph \( G \). In contrast, the moment method of Kwan and Wigderson, as well as the orthogonal rank bound in~\eqref{eq:existing-bounds}, can establish that \( n_{\geq 0}(G) = n^{1 - o(1)} \) for certain graphs \( G \), at the cost of much stronger assumptions on $G$.

% In contrast, the moment method of Kwan and Widgerson , as well as the fractional rank method, can show that $n_{\geq 0}(G) = n^{1 - o(1)}$ for some graphs $G$, at the cost of much stronger assumptions on $G$.

% We remark that due to the Alon--Boppana bound $\lambda = \Omega(d^{1/2})$, Theorem~\ref{thm:ndlambda1} cannot show bounds better than $n_{\geq 0}(G) = \Omega(n^{1/6})$ for any graph $G$. 

We denote the \emph{Shannon capacity} by \( \Theta(G) \), and the \emph{Lov{\'a}sz theta function} by \( \vartheta(G) \); for a comprehensive introduction, see~\cite[Section~11]{Lovasz19}. It is proved in \cite[Eq.~(11.56) and Proposition~11.31]{Lovasz19} that \( \alpha(G) \le \Theta(G) \le \vartheta(G) \) for all graphs \( G \), thereby establishing the Shannon capacity and the Lov{\'a}sz theta function as powerful upper bounds on the independence number. Since \( \alpha(G) \le n_{\geq 0}(G) \) also holds by the inertia bound, it is natural to investigate the relationship between \( n_{\geq 0}(G) \), \( \Theta(G) \), and \( \vartheta(G) \).

A result of Alon~\cite[Theorem~3.2]{Alon2019} showed that \( \vartheta(G) \) is not upper bounded by any function of \( \Theta(G) \). As another application of Theorem~\ref{thm:ndlambda1}, we provide a parallel result showing that $n_{\geq 0}(G)$ cannot be upper bounded by any function of $\Theta(G)$.
\begin{thm}
\label{thm:main-shannon}
There exists a sequence of graphs $H_n$ on $n$ vertices such that $\Theta(H_n) \leq 3$ and $n_{\geq 0}(H_n) = \Omega(n^{1/12})$.
\end{thm}
The works~\cite{Kwan2024} and~\cite{Ihringer2023} together demonstrate that the inertia and the theta function are incomparable, that is, each can be strictly greater than the other. Specifically, in the first paragraph and footnote of~\cite[p.~3200]{Kwan2024}, Kwan and Wigderson apply~\cite[Theorem~1.7(ii)]{Kwan2024} by taking \( G \) to be an edge-transitive induced subgraph on at least half of the vertices of the polarity graph of a finite projective plane with girth at least 5. This yields the existence of a family of graphs \( G \) such that
\[
n_{\geq 0}(G) = \Omega(|G|), \qquad \vartheta(G) = O(|G|^{3/4}).
\] On the other hand, Ihringer~\cite{Ihringer2023} showed the existence of a family of graphs \( G \) for which
\[
n_{\geq 0}(G) = O(|G|^{1/2}), \qquad \vartheta(G) = \Omega(|G|^{3/4}).
\] While previous works show incomparability, they do not exclude the possibility of a polynomial relationship between $\vartheta(G)$ and $n_{\geq 0}(G)$. In the following theorem, we show that $\vartheta(G)$ can be exponentially larger than $n_{\geq 0}(G)$.
\begin{thm}\label{thm:theta1}
    For each positive integer $k \geq 2$, there exists a graph $G$ with
    $$n_{\geq 0}(G) \leq k^3 + 1, \qquad \vartheta(G) \geq  2^k.$$
\end{thm}
% We leave it as an open problem to determine whether \( n_{\geq 0}(G) \) can be bounded above or below by any function of \( \vartheta(G) \).

The rest of the paper is organized as follows. In Section~\ref{sec:expanders}, we prove Theorem~\ref{thm:ndlambda1} and establish Theorems~\ref{thm:main} and~\ref{thm:main-shannon} as applications. Section~\ref{sec:lovasztheta} explores the relationship between the inertia and the Lov\'{a}sz theta function, and establishes Theorem~\ref{thm:theta1}. Finally, Section~\ref{sec:conclusion} concludes with several open problems for future research. The proofs of Propositions~\ref{prop:large-indep} and~\ref{prop:projective-independence-2} are deferred to Appendix~\ref{sec:orthogonal-rank}, since they mostly follow from existing results.

\section*{Acknowledgements}
This material is based upon work supported by the National Science
Foundation under Grant No. DMS-1928930, while the second author was in
residence at the Simons Laufer Mathematical Sciences Institute in
Berkeley, California, during the Spring 2025 semester. This material is also partially supported by NSF Award DMS-2154129. The second author would like to thank Matija Buci\'{c} for organizing the workshop \emph{Expander graphs and their applications} at Princeton University, and thank Igor Balla and Yuval Widgerson for insightful discussions.

\section{Inertia bound via expanders}
\label{sec:expanders}
In this section, we first prove Theorem~\ref{thm:ndlambda1}, and then deduce Theorems~\ref{thm:main} and~\ref{thm:main-shannon} as corollaries. As a first step, we employ a key insight which originated
from works on matrix scaling by Csima and Datta \cite{CsimaDatta1972, Datta1970} and was central 
in~\cite[Section~3]{Kwan2024}. This observation allows us to restrict our attention to weighted adjacency matrices whose rows all have the same \( L^2 \)-norm. 
\begin{defn}
A Hermitian matrix \( A \in \mathbb{C}^{n \times n} \) is said to be \emph{1-regular} if each of its rows has \( L^2 \)-norm equal to 1; that is,
\[
\sum_{j=1}^n |A_{ij}|^2 = 1 \quad \text{for all } i \in [n].
\] 
Here we write \( [n] = \{1, \dots, n\} \).
\end{defn}
The following lemma is essentially identical to the argument in~\cite{Kwan2024}, with minor adjustments to fit our setting. If \( M \) is an \( n \times n \) matrix and \( S, T \subseteq [n] \), we write \( M[S, T] \) to denote the submatrix of \( M \) consisting of the rows indexed by \( S \) and the columns indexed by \( T \). Recall that a matrix is said to be \emph{doubly stochastic} if it has nonnegative
entries and the sum of the entries in every row and every column is equal to~1.

\begin{lem}\label{lem:make-em-regular}
Let \( G = (V, E) \) be a graph on \( n \) vertices. Suppose that for any vertex sets \( S, T \subseteq V \) with no edges between them, we have \( |S| + |T| < n \). Then, for any weighted adjacency matrix \( A \) of \( G \), there exists a \( 1 \)-regular weighted adjacency matrix \( B \) of \( G \) such that
\[
n_{\geq 0}(B) \leq n_{\geq 0}(A).
\]
\end{lem}
\begin{proof}Since the eigenvalues of a Hermitian matrix vary continuously under perturbation, the number of nonnegative eigenvalues, \( n_{\geq 0}(\cdot) \), is an upper semicontinuous function. Thus, there exists a neighbor $U$ of $A$ in the space of Hermitian matrices such that $n_{\geq 0}(A') \leq n_{\geq 0}(A)$ for every $A' \in U$. It is clear that we can find an $A' \in U$ such that $A'$ is a weighted adjacency matrix of $G$, and $A'_{ij} \neq 0$ for every edge $ij \in E$ of $G$. Let $M$ be the symmetric matrix with $M_{ij} = \abs{A'_{ij}}^2$.

If \( S, T \subseteq V \) are nonempty subsets with \( M[S, T] = 0 \), then there are no edges between \( S \) and \( T \) in \( G \), so \( |S| + |T| < n \) by assumption. By~\cite[Lemma~2.7]{Kwan2024}, there exists a diagonal matrix \( D \) with strictly positive diagonal entries such that \( D M D \) is doubly stochastic. Now we define \( B = D^{1/2} A' D^{1/2} \). Then for all \( i, j \), we have
\[
|B_{ij}|^2 = \left| \sqrt{D_{ii}}\, A'_{ij} \sqrt{D_{jj}} \right|^2 = (D M D)_{ij},
\]
so \( B \) is a \( 1 \)-regular weighted adjacency matrix of \( G \). Furthermore, by~\cite[Lemma~2.4]{Kwan2024}, we have
\[
n_{\geq 0}(B) = n_{\geq 0}(A') \leq n_{\geq 0}(A),
\]
as desired. \end{proof}

Recall that for any matrix \( X \in \mathbb{C}^{m \times n} \), we denote by \( X^* \) the \emph{conjugate transpose} of \( X \). The \emph{Frobenius inner product} between matrices \( X \) and \( Y \) is defined by \( \langle X, Y \rangle = \tr(X^* Y) \). For any vector \( \mathbf{x} = (x_1, \dots, x_n) \in \mathbb{C}^n \) and any \( p \geq 1 \), the \emph{\( p \)-norm} of \( \mathbf{x} \) is defined by
\[
\| \mathbf{x} \|_p = \left( \sum_{i=1}^n |x_i|^p \right)^{1/p}.
\]
We are now ready to present the following

\begin{proof}[Proof of Theorem~\ref{thm:ndlambda1}] Since \( \Gamma \) is a spanning subgraph of the complement of \( G \), we have \( G \subseteq \overline{\Gamma} \). As \( n_{\geq 0}(G) \geq n_{\geq 0}(\overline{\Gamma}) \), it suffices to prove the desired bound for \( G = \overline{\Gamma} \). Hence, we may assume without loss of generality that \( G \) is the complement of \( \Gamma \). Let \( \delta(G) \) denote the minimum vertex degree of \( G \). Then $\delta(G) = n - d - 1 > \frac{n}{2}$. For any non-empty vertex sets $S, T \subset V$ with no edges between them, we have $\abs{S}, \abs{T} \leq n - \delta(G) < \frac{n}{2}$, so $\abs{S} + \abs{T} < n$. 
    
    Let $A_0$ be any weighted adjacency matrix of $G$. By Lemma~\ref{lem:make-em-regular}, there exists a weighted adjacency matrix $A$ of $G$ such that $A$ is $1$-regular and $n_{\geq 0}(A) \leq n_{\geq 0}(A_0)$. So it suffices to prove the lower bound for $n_{\geq 0}(A)$. Recall that each row of $A$ has $L^2$-norm $1$, i.e.
    $$\sum_{j \in V} \abs{A_{ij}}^2 = 1,\quad \forall i \in V.$$
    Let $k = n_{\geq 0}(A)$. Let $\lambda_1(A) \geq \cdots \geq \lambda_n(A)$ be the eigenvalues of $A$. We first estimate the first eigenvalue $\lambda_1(A)$. Since $A$ is Hermitian with zero diagonal entries, we have
\[
\sum_{i = 1}^n \lambda_i(A) = \sum_{i \in V} A_{ii} = 0, \quad \text{and} \quad \sum_{i=1}^n \lambda_i(A)^2 = \sum_{i,j \in V} |A_{ij}|^2 = n.
\] Therefore, we obtain
    $$\sum_{i = 1}^k \lambda_i(A) = \frac{1}{2}\sum_{i = 1}^n \abs{\lambda_i(A)} \geq \frac{1}{2}\sqrt{\sum_{i = 1}^n \lambda_i(A)^2} = \frac{\sqrt{n}}{2},$$ which gives
    $$\lambda_1(A) \geq \frac{\sqrt{n}}{2k}.$$ Next, let \( \mathbf{v} \) be the normalized eigenvector of $A$ corresponding to \( \lambda_1(A) \).  
We give a bound on the \( \infty \)-norm of \( \mathbf{v} \).  
For each vertex \( i \in V \), let \( v_i \) denote the $i$th entry of \( \mathbf{v} \).  
Then
\[
v_i = \frac{1}{\lambda_1(A)} \sum_{j \in V} A_{ij} v_j.
\] By the Cauchy-Schwarz inequality, we have
    $$\abs{v_i} \leq \frac{1}{\lambda_1(A)} \left(\sum_{j \in V} \abs{A_{ij}}^{2}\right)^{1/2} \norm{\bv}_2 \leq \frac{1}{\lambda_1(A)} \leq \frac{2k}{\sqrt{n}},$$
    so we obtain $\norm{\bv}_\infty \leq \frac{2k}{\sqrt{n}}$.
    
    Now let $A_{\Gamma}$ be the unweighted adjacency matrix of $\Gamma$, and let $\Lambda$ be the diagonal matrix with diagonal entries $\Lambda_{ii} = v_i$ for each $i \in V$. We consider the Frobenius inner product
\[
\langle A, \Lambda (\lambda I - A_{\Gamma}) \Lambda^* \rangle.
\] On one hand, as $\Gamma$ is contained in the complement of $G$, we have $A_{ij} \neq 0$ implies $(\lambda I - A_{\Gamma})_{ij} = 0$. Therefore, we must have
    $$\langle A, \Lambda (\lambda I - A_{\Gamma}) \Lambda^* \rangle = 0.$$
    On the other hand, let $\bv^i$ denote the normalized eigenvector of $A$ corresponding to $\lambda_i(A)$, with $\bv^1 = \bv$. We consider the spectral decomposition
    $$A = \sum_{i = 1}^n \lambda_i(A) \bv^i {\bv^i}^*.$$
    Therefore, we have
    $$\sum_{i = 1}^n \lambda_i(A) \langle \bv^i {\bv^i}^*, \Lambda (\lambda I - A_{\Gamma}) \Lambda^* \rangle = 0.$$
    In other words, we have
    \begin{equation}\label{eq:ndlambdainertiabound1}
    \sum_{i = 2}^n \lambda_i(A) \langle \bv^i {\bv^i}^*, \Lambda (\lambda I - A_{\Gamma}) \Lambda^* \rangle = \lambda_1 \langle \bv \bv^*, \Lambda (A_\Gamma - \lambda I) \Lambda^* \rangle.\end{equation}
    For complex vectors $\bv = (v_1, \dots, v_n) \in \C^n$ and $\bw = (w_1, \dots, w_n) \in \C^n$, let $\bv \circ \bw$ denote the vector in $\C^n$ whose entries are
    $$(\bv \circ \bw)_i = v_i w_i^*.$$
    We compute that
    \begin{align*}
        \langle \bv^i {\bv^i}^*, \Lambda (\lambda I - A_{\Gamma}) \Lambda^* \rangle  
        =& \tr(\bv^i {\bv^i}^* \Lambda (\lambda I - A_{\Gamma}) \Lambda^*) \\
        =& {\bv^i}^* \Lambda (\lambda I - A_{\Gamma}) \Lambda^* \bv^i \\
        =&  (\bv^i \circ \bv)^* (\lambda I - A_{\Gamma}) (\bv^i \circ \bv).  
    \end{align*}
    For any \( i \geq 2 \), we have
\[
\langle \bv^i \circ \bv, \mathbf{1} \rangle = \langle \bv^i , \bv \rangle = 0.
\]
Therefore, \( \bv^i \circ \bv \) is orthogonal to \( \mathbf{1} \), the all-one vector. Since \( \Gamma \) is a \( d \)-regular graph, the largest eigenvalue of its adjacency matrix \( A_\Gamma \) is \( d \), and the corresponding eigenvector is parallel to \( \mathbf{1} \). Hence, \( \bv^i \circ \bv \) lies entirely in the orthogonal complement of the top eigenspace of \( A_\Gamma \). As \( \Gamma \) is an \( (n, d, \lambda) \)-graph, all other eigenvalues of \( A_\Gamma \) lie in the interval \( [-\lambda, \lambda] \). Hence, we have
\[
0 \leq (\mathbf{v}^i \circ \mathbf{v})^* (\lambda I - A_\Gamma) (\mathbf{v}^i \circ \mathbf{v}) \leq 2 \lambda \| \mathbf{v}^i \circ \mathbf{v} \|_2^2.
\] Furthermore, we have
    $$\norm{\bv^i \circ \bv}_2^2 \leq \norm{\bv^i}_2^2 \norm{\bv}_\infty^2 \leq \frac{4k^2}{n},$$ since $\norm{\bv}_\infty \leq \frac{2k}{\sqrt{n}}$. Therefore, we obtain
    \begin{equation}\label{eq:ndlambdainertiabound2}\sum_{i = 2}^n \lambda_i(A)\langle \bv^i {\bv^i}^*, \Lambda (\lambda I - A_{\Gamma}) \Lambda \rangle \leq \sum_{i = 2}^k 2 \lambda \lambda_i(A)\norm{\bv^i \circ \bv}_2^2 \leq \frac{8k^2 \lambda}{n} \sum_{i = 2}^k \lambda_i(A).\end{equation}
    On the other hand, we have
    $$\langle \bv \bv^*, \Lambda (A_\Gamma - \lambda I) \Lambda \rangle = (\bv \circ \bv)^* (A_{\Gamma} - \lambda I) (\bv \circ \bv).$$
    As $\langle \bv \circ \bv, \mathbf{1}\rangle = \langle \bv, \bv \rangle = 1$, we have $\bv \circ \bv = \frac{1}{n} \mathbf{1} + \bx$ for some vector $\bx$ orthogonal to $\mathbf{1}$. We compute that
    $$(\bv \circ \bv)^* (A_{\Gamma} - \lambda I) (\bv \circ \bv) = \frac{d - \lambda}{n} +  \bx^*(A_{\Gamma} - \lambda I)\bx.$$
    As $\bx^*(A_{\Gamma} - \lambda I)\bx \geq -2\lambda \norm{\bx}_2^2$, we get
    $$(\bv \circ \bv)^* (A_{\Gamma} - \lambda I) (\bv \circ \bv) \geq \frac{d - \lambda}{n} - 2\lambda \norm{\bx}_2^2 = \frac{d - \lambda}{n} - 2\lambda \left(\norm{\bv \circ \bv}_2^2 - \frac{1}{n}\right).$$
    Simplifying, we obtain
    $$
    (\bv \circ \bv)^* (A_{\Gamma} - \lambda I) (\bv \circ \bv) \geq \frac{d + \lambda}{n} - 2\lambda \norm{\bv \circ \bv}_2^2 \geq \frac{d + \lambda}{n} - \frac{8k^2 \lambda}{n}.$$
    Therefore, we conclude that
    \begin{equation}\label{eq:ndlambdainertiabound3}\lambda_1(A) \langle \bv \bv^*, \Lambda (A_\Gamma - \lambda I) \Lambda \rangle \geq \frac{d + \lambda}{n} \lambda_1(A) -  \frac{8k^2 \lambda}{n} \lambda_1(A).\end{equation}
    Combining~\eqref{eq:ndlambdainertiabound1}, \eqref{eq:ndlambdainertiabound2}, and~\eqref{eq:ndlambdainertiabound3}, we obtain
    $$\frac{d + \lambda}{n} \lambda_1(A) -  \frac{8k^2 \lambda}{n} \lambda_1(A) \leq \frac{8k^2\lambda}{n} \sum_{i = 2}^k \lambda_i(A).$$
    Thus we get
    $$\frac{d + \lambda}{8k^2 \lambda}\leq \frac{\sum_{i = 1}^k \lambda_i(A)}{\lambda_1(A)}.$$
    Finally, as $\lambda_1(A)$ is the greatest eigenvalue of $A$, we obtain
    $$\frac{\sum_{i = 1}^k \lambda_i(A)}{\lambda_1(A)} \leq k.$$
    So we conclude that 
    $$\frac{d + \lambda}{8k^2 \lambda} \leq k.$$
    Simplifying, we have
    $$k \geq \left(\frac{d + \lambda}{8\lambda}\right)^{1/3},$$
    as desired.
\end{proof}
Theorem~\ref{thm:main} now follows as a corollary.
\begin{proof}[Proof of Theorem~\ref{thm:main}]
    A classical construction of Alon in \cite[Theorem~2.1]{Alon1994} shows that for every positive integer $n = 2^{3k}$ with $k$ not divisible by $3$, there exist a triangle-free $(n, d, \lambda)$-graph $\Gamma_n$ with $d = 2^{k - 1}(2^{k - 1} - 1) = \Omega(n^{2/3})$ and $\lambda = O(d^{1/2})$.\footnote{We note that recently Kim and Lee \cite{KimLee2024} gave a similar construction for all orders $n$, though using their result would weaken our theorem by a logarithmic factor.} It is clear that $d < \frac{n}{2} - 1$. Let $\overline{\Gamma_n}$ be the complement of $\Gamma_n$. Applying Theorem~\ref{thm:ndlambda1} yields
    \[
    n_{\geq 0}(\overline{\Gamma_n}) = \Omega((d / \lambda)^{1/3}) = \Omega(n^{1/9}).
    \]

For \( n \geq 8 \), let \( n_0 \) be the largest positive integer of the form \( n_0 = 2^{3k} \), where \( k \) is not divisible by 3 and \( n_0 \leq n \). Let \( G_n \) be the graph obtained by adding \( (n - n_0) \) universal vertices to \( \overline{\Gamma_{n_0}} \). Then we have \( \alpha(G_n) \leq 2 \). Furthermore, by the Cauchy interlacing theorem (see~\cite[Lemma~2.5]{Kwan2024}), it follows that
\[
n_{\geq 0}(G_n) \geq n_{\geq 0}(\overline{\Gamma_{n_0}}) = \Omega(n_0^{1/9}) = \Omega(n^{1/9}),
\]
as desired.

For $n < 8$, we may simply take $G_n = K_n$, which satisfies $\alpha(G_n) = 1 \leq 2$ and trivially $n_{\geq 0}(G_n) = 1$.\end{proof}

The proof of Theorem~\ref{thm:main-shannon} is similar.
\begin{proof}[Proof of Theorem~\ref{thm:main-shannon}]
We consider the graphs $H_n$ constructed by Alon in~\cite[Theorem~3.2]{Alon2019}. Alon showed that $\Theta(H_n) \leq 3$, and moreover, the complement $\overline{H_n}$ is an induced subgraph of the polarity graph of the projective plane of order $q$.  From the proof of the second claim in~\cite[Theorem~3.2]{Alon2019}, the second-largest eigenvalue of the polarity graph is $\sqrt{q}$ and the smallest eigenvalue is $-\sqrt{q}$. By Cauchy interlacing, we obtain
\[
\max\{ |\lambda_2(\overline{H_n})|, |\lambda_n(\overline{H_n})| \} \leq \sqrt{q}.
\]
The proof of Alon also shows that $\overline{H_n}$ has $n = q^2 - 1$ vertices and is $q$-regular. Thus, $\overline{H_n}$ is an \((n, d, \lambda)\)-graph with parameters \(n = q^2 - 1\), \(d = q\), and \(\lambda = \sqrt{q}\). Applying Theorem~\ref{thm:ndlambda1} yields
\[
n_{\geq 0}(H_n) \geq \left(\frac{q + \sqrt{q}}{8 \sqrt{q}}\right)^{1/3} = \Omega(q^{1/6}) = \Omega(n^{1/12}),
\]
as desired.
\end{proof}

\section{Inertia and Lov\'{a}sz's theta function}
\label{sec:lovasztheta}
%A well-known property of the Lov\'{a}sz number is that \( \alpha(G) \leq \vartheta(G) \), while the inertia bound implies that \( \alpha(G) \leq n_{\geq 0}(G) \). It is also known that the Lov\'{a}sz number satisfies the inequality
%\begin{equation} \label{eq:theta2}
%\vartheta(G) \geq \frac{|G|}{\vartheta(\overline{G})}.
%\end{equation}
%In view of the fact that both the ratio bound and the square energy bound are valid lower bounds for the Lov\'{a}sz theta function, one might naturally conjecture that
%\begin{equation} \label{eq:theta1}
%n_{\geq 0}(G) \geq \frac{|G|}{\vartheta(\overline{G})}.
%\end{equation}

%However, this conjecture is not true for $G = C_5$. For the all  $-1$ negative weight matrix, we have $n_{\geq 0}(G) = 2$, yet $\frac{n}{\vartheta(\Bar{G})} = \frac{5}{\sqrt{5}} = \sqrt{5}$. 

%Here is another example to show that $n_{\geq 0}(G)$ can be exponentially smaller than $\frac{n}{\vartheta(\Bar{G})}$. See \url{https://arxiv.org/pdf/2402.05818} \cite{Linz2025}.
In this section, we construct graphs such that $\vartheta(G)$ is exponentially larger than $n_{\geq 0}(G)$. Our construction uses the following family of graphs with close relation to the forbidden intersection problem in extremal set theory.
\begin{defn}
Let \( n \) and \( k \) be positive integers with \( n > k \), and let \( L \subseteq \{0, 1, \dots, k - 1\} \). The \emph{generalized Johnson graph} $G = G(n, k, L)$ is the graph with $V(G) = \binom{[n]}{k}$, and $AB \in E(G)$ if and only if $\abs{A \cap B} \notin L$.
\end{defn}
%For example, the clique number of $G(n, k, L)$ corresponds to the restricted intersection problem in extremal set theory. Frankl and F\"{u}redi \cite{Frankl1985} determined the clique number of $G(n, k, \{\ell\})$ for sufficiently large $n$. 
Recently, Linz \cite{Linz2025} asymptotically determined the Lov\'{a}sz theta function of $G(n, k, L)$. We only need the special case where $L = \{\ell\}$ here. 
\begin{thm}[\cite{Linz2025}, Theorem 1.5]
Let $n$ and $k$ be positive integers with $n > k$ and $\ell$ be an integer in $[0, k - 1]$. Then we have
$$\vartheta(G(n, k, \{\ell\})) = \frac{\binom{k}{k - \ell}}{(k - \ell)(\ell + 1)} n + O_k(1).$$
\end{thm}
We now carefully go over Linz's proof of this theorem to make the dependency on $k$ explicit. 
\begin{prop}
\label{thm:johnson-theta}
Let $n$ and $k$ be positive integers with $k \geq 2$ and $n \geq k^3 + 1$. Let $\ell$ be an integer in $[0, k - 1]$. Then we have
$$\vartheta(G(n, k, \{\ell\})) \geq \frac{\binom{k}{k - \ell}}{(k - \ell)(\ell + 1)} (n - k).$$
\end{prop}
\begin{proof}
    In the proof of Theorem~1.5 in \cite[Page~7]{Linz2025}, Linz gave the following explicit characterization of $\vartheta(G(n, k, \{\ell\}))$: 
    $$\vartheta(G(n, k, \{\ell\})) = \max(1 + a_{k - \ell}),$$
    subject to the constraint that for each $u \in \{0, 1, \dots, k\}$
    \begin{equation}
        \label{eq:theta-constraint}
     a_{k - \ell} \frac{1}{\binom{k}{k - \ell} \binom{n - k}{k - \ell}} \sum_{j = 0}^{k - \ell} (-1)^j \binom{u}{j} \binom{k - u}{k - \ell - j} \binom{n - k - u}{k - \ell - j} \geq -1.   
    \end{equation}
Let \( f_{u, j} \) denote the product terms appearing in this constraint
    $$f_{u, j} = \binom{u}{j} \binom{k - u}{k - \ell - j} \binom{n - k - u}{k - \ell - j}.$$
    Clearly, we have $f_{u, j} \neq 0$ only if $j \in [u - \ell, u]$. We now show that $f_{u, j}$ decays in this interval. 

\begin{claim}\label{claim1} For any $j \geq \max(0, u - \ell)$ and $j \leq k - \ell$, we have $f_{u, j} \geq f_{u, j + 1}$.
\end{claim}
\begin{proof}[Proof of Claim~\ref{claim1}]
We have $f_{u, j + 1} = 0$ for $j > u$, so we may assume that $j \leq u$. We can explicit compute the ratio as
    $$\frac{f_{u, j + 1}}{f_{u, j}} = \frac{u - j}{j + 1} \cdot \frac{k - \ell - j}{j + \ell - u + 1} \cdot \frac{k - \ell - j}{n - u - \ell + j + 1}.$$
    As $j \in [u - \ell, k - \ell]$ and $n \geq k^3 + 1$, we have
    $$(j + \ell - u + 1)(n - u - \ell + j + 1) \geq n - 1.$$
    Therefore, we obtain
    $$\frac{f_{u, j + 1}}{f_{u, j}} \leq \frac{u (k - \ell)^2}{n - 1} \leq \frac{k^3}{n - 1} \leq 1,$$
    which proves Claim~\ref{claim1}.    
\end{proof}

    We now show that
    $$a_{k - \ell} = \frac{\binom{k}{k - \ell}}{(k - \ell)(\ell + 1)} (n -  k)$$
    satisfies \eqref{eq:theta-constraint}, which establishes this proposition. We divide into two regimes.

    \textbf{Case 1: }$u \leq \ell$. In this case, we have $f_{u, j} \geq f_{u, j + 1}$ for every $j \in [0, k - \ell]$. Therefore, we have
    $$\sum_{j = 0}^{k - \ell} (-1)^j \binom{u}{j} \binom{k - u}{k - \ell - j} \binom{n - k - u}{k - \ell - j} = \sum_{j = 0}^{k - \ell} (-1)^j f_{u, j} \geq 0,$$
    so the left hand side of \eqref{eq:theta-constraint} is non-negative.

    \textbf{Case 2: }$u > \ell$. In this case, the first non-zero term in the sum is $j = u - \ell$, and $f_{u, j} \geq f_{u, j + 1}$ for every $j \in [u - \ell, k - \ell]$. Therefore, we have
    $$\sum_{j = 0}^{k - \ell} (-1)^j \binom{u}{j} \binom{k - u}{k - \ell - j} \binom{n - k - u}{k - \ell - j} \geq -f_{u, u - \ell} = -\binom{u}{u - \ell} \binom{n - k - u}{k - u}.$$
    Thus, it suffices to check that
    \begin{equation}
        \label{eq:reduced-constraint}
     a_{k - \ell} \frac{1}{\binom{k}{k - \ell} \binom{n - k}{k - \ell}} \binom{u}{u - \ell} \binom{n - k - u}{k - u} \leq 1.   
    \end{equation}
    We observe that for any $u \in [\ell, k - 1]$
    $$\frac{\binom{u + 1}{u + 1 - \ell} \binom{n - k - u - 1}{k - u - 1} }{\binom{u}{u - \ell} \binom{n - k - u}{k - u}} = \frac{u + 1}{u + 1 - \ell} \cdot \frac{k - u}{n - k - u} \leq \frac{k^2}{n - 2k} \leq 1.$$
    Therefore, the left hand side of \eqref{eq:reduced-constraint} is non-increasing in $u$, and it suffices to check \eqref{eq:reduced-constraint} for $u = \ell + 1$. This simplifies to
    $$ a_{k - \ell} \frac{1}{\binom{k}{k - \ell} \binom{n - k}{k - \ell}} (\ell + 1) \binom{n - k - \ell - 1}{k - \ell - 1} = \frac{\binom{n - k - \ell - 1}{k - \ell - 1}}{\binom{n - k -  1}{k - \ell - 1}} \leq 1,$$
    which follows directly from the definition of $a_{k - \ell}$.
\end{proof}
On the other hand, the inertia of $G(n, k, \{\ell\})$ is small.
\begin{prop}
    \label{prop:johnson-inertia}
    For any positive integers $n, k, \ell$ with $\ell \in [0, k - 1]$, we have
    $$n_{\geq 0}(G(n, k, \{\ell\})) \leq n.$$
\end{prop}
\begin{proof}
    We consider the weighted adjacency matrix \( M \) of \( G(n, k, \{\ell\}) \), defined by
\[
M_{A, B} = 
\begin{cases}
|A \cap B| - \ell, & \text{if } A \ne B, \\
0, & \text{if } A = B.
\end{cases}
\] Let \( I \) denote the identity matrix, \( J \) the all-one matrix, and \( U \) the \( \binom{n}{k} \times n \) matrix whose row corresponding to each \( A \in \binom{[n]}{k} \) is the indicator vector of \( A \). Then the \((A, B)\)-entry of \( UU^\top \) is given by \((UU^\top)_{A,B} = |A \cap B|\). Therefore, the matrix \( M \) can be written as
\[
M = UU^\top - (k - \ell) I - \ell J.
\] Since $-(k - \ell) I - \ell J$ is strictly negative-definite, and $UU^\top$ has at most $n$ non-zero eigenvalues, we conclude that $M$ has at most $n$ non-negative eigenvalues.
\end{proof}
Combining the above two results gives us Theorem~\ref{thm:theta1}.
\begin{proof}[Proof of Theorem~\ref{thm:theta1}]
    For each positive integer $k \geq 2$, set $n = k^3 + 1$ and $\ell = \lfloor \frac{k}{2} \rfloor$. Take $G$ to be the generalized Johnson graph $G = G(n, k, \ell)$. Then Proposition~\ref{prop:johnson-inertia} shows that
    $$n_{\geq 0}(G) \leq n = k^3 + 1.$$ It is well known that \(\binom{k}{\lfloor \frac{k}{2} \rfloor} \geq \frac{2^k}{\sqrt{2k}}\). Thus, by Proposition~\ref{thm:johnson-theta}, we obtain
\[
\vartheta(G) 
\geq \frac{\binom{k}{k - \ell}}{(k - \ell)(\ell + 1)} (n - k) 
\geq \frac{\binom{k}{\lfloor \frac{k}{2} \rfloor}}{\left(\frac{k + 1}{2}\right)^2} (k^3 - k + 1) 
\geq \frac{4(k^3 - k + 1)}{\sqrt{2k} \cdot (k + 1)^2}   2^k 
\geq  2^k,
\]as desired.
\end{proof}

%Since \( \alpha(G) \leq \vartheta(G) \) and \( \alpha(G) \leq n_{\geq 0}(G) \). Naturally, one might ask: is there a strength relationship between these two bounds?

%The cycle \( C_5 \) shows that there exists a graph \( G \) such that \( n_{\geq 0}(G) < \vartheta(G) \).

%\begin{ques} Does there exist a graph \( G \) such that \( n_{\geq 0}(G) > \vartheta(G) \)?
%\end{ques}

%[\cite{Kwan2024} shows that for $C_4$-free graphs $n_{\geq 0}(G) = \Theta(|G|)$. Let $\FF_q$ be a finite field. Then the orthogonality graph $G_q$ on $\PP\FF_q^2$ has $q^2 + q + 1$ vertices, is vertex-transitive, and satisfies (See page $9$ of \url{https://arxiv.org/pdf/1905.01539}.)
%$$\vartheta(\overline{G_q}) = \Omega(|G_q|^{1/4}).$$
%By vertex transitivity, we have
%$$\vartheta(G_q) = \frac{|G_q|}{\vartheta(\overline{G_q})} = O(|G_q|^{3/4}).$$
%So we have $n_{\geq 0}(G_q) = \Omega(|G_q|)$ and $\vartheta(G_q) = O(|G_q|^{3/4})$.
%]

\section{Concluding Remarks}
\label{sec:conclusion}
We believe that our results in this paper are just the tip of an iceberg, and many other interesting things can be said about the inertia bound. In this section, we discuss several open problems that may serve as motivation for future work.

\subsection{Inertia, Shannon capacity, and Lov\'{a}sz theta function}

Theorem~\ref{thm:main-shannon} shows that \( n_{\geq 0}(G) \) cannot be upper bounded by any function of \( \Theta(G) \); however, it remains unclear whether the converse holds. In addition, Theorem~\ref{thm:theta1} implies that there exist graphs \( G \) with arbitrarily large \( n_{\geq 0}(G) \) for which \( \vartheta(G) \geq 2^{n_{\geq 0}(G)^{1/3} - 1} \). Nevertheless, this does not rule out the possibility that \( \vartheta(G) \) is upper bounded by some function of \( n_{\geq 0}(G) \), or conversely, that \( n_{\geq 0}(G) \) is upper bounded by some function of \( \vartheta(G) \). We leave these relationships as an open question.

\begin{ques}
    1) Is $\Theta(G)$ upper-bounded by some function of $n_{\geq 0}(G)$? 

    2) Is $\vartheta(G)$ upper-bounded by some function of $n_{\geq 0}(G)$? 
    
    3) Conversely, is $n_{\geq 0}(G)$ upper-bounded by some function of $\vartheta(G)$? 
    
    4) More strongly, is $n_{\geq 0}(G)$ upper-bounded by $\vartheta(G)^C$ for some constant $C > 0$?
\end{ques}

% \subsection{Inertia and theta function}
% Theorem~\ref{thm:theta1} implies that there exist graphs \( G \) with arbitrarily large $n_{\geq 0}(G)$ for which \( \vartheta(G) \) is at least \( 2^{n_{\geq 0}(G)^{\frac{1}{3}} - 1} \). However, this still does not rule out the possibility that \( \vartheta(G) \) is upper-bounded by some function of \( n_{\geq 0}(G) \), or conversely, whether \( n_{\geq 0}(G) \) is upper-bounded by some function of \( \vartheta(G) \). We leave these relations as an open question.

\subsection{Nordhaus--Gaddum type bounds for inertia}
\label{sec:NordhausGaddum}
%Shengtong: I have added this section back, as Ihringer seems to like this conjecture, and it would be bad for others if they are already working on it and we removed it. I also added a new open--ended question. As we do not prove many new things here, I put it in the concluding remarks.

The original Nordhaus--Gaddum inequalities state that for any graph \( G \) on \( n \) vertices,
\begin{equation}\label{eq:classicNordhausGaddum1}
\chi(G) \cdot \chi(\overline{G}) \geq n
\quad \text{and}  \quad \chi(G) + \chi(\overline{G}) \leq n + 1.
\end{equation}
These classic inequalities have inspired a wide range of Nordhaus--Gaddum type inequalities for other graph parameters. For example, the survey paper by Aouchiche and Hansen~\cite{AouchicheHansen2017} reviews scores of Nordhaus--Gaddum type inequalities. Notably, the Lov\'{a}sz theta function also satisfies Nordhaus--Gaddum type bounds:
\[
\vartheta(G) \cdot \vartheta(\overline{G}) \geq n
\quad \text{and} \quad
\vartheta(G) + \vartheta(\overline{G}) \leq n + 1.
\]
The inequality \( \vartheta(G) \cdot \vartheta(\overline{G}) \geq n \) is given in~\cite[Proposition~11.4]{Lovasz19}. The second inequality is immediate from the well-known relation \( \vartheta(G) \leq \chi(\overline{G}) \) and \eqref{eq:classicNordhausGaddum1}.

It is natural to ask whether the inertia of a graph also satisfies a similar Nordhaus--Gaddum type bound. The multiplicative bound does not hold as stated for \( n_{\geq 0}(G) \). For example, let \( G = C_5 \). The negated \(\{0, 1\}\)-adjacency matrix of \( G \) has only two non-negative eigenvalues, so we have \(n_{\geq 0}(G) = n_{\geq 0}(\overline{G}) = 2\), and hence \(n_{\geq 0}(G) \cdot n_{\geq 0}(\overline{G}) = 4 < 5 = n.\) We do not know if a weaker version of the Nordhaus--Gaddum bound can hold.
\begin{ques}
 \label{problem:NordhausGaddumforinertia0}
What is the optimal lower bound for $ n_{\geq 0}(G) \cdot n_{\geq 0}(\overline{G}) $ in terms of the number of vertices $n$? Is it polynomial in $n$? 
\end{ques}
We remark that the inertia bound implies
\[
n_{\geq 0}(G) \cdot n_{\geq 0}(\overline{G}) \geq \alpha(G) \cdot \alpha(\overline{G}) \geq \mu_n,
\]
where $\mu_n = \min\{ ab : R(a + 1, b + 1) > n \} = O(\log^2 n)$ is the classical Nordhaus--Gaddum type lower bound for the independence number, see \cite[Theorem~2.100]{AouchicheHansen2017} or \cite{ChartrandSchuster1974}. Moreover, for certain small graphs, equality is achieved. For instance, when $ G = C_4 $ or $ C_5 $, we have \(n_{\geq 0}(G) \cdot n_{\geq 0}(\overline{G}) = \mu_n\). However, it remains unclear whether this lower bound is tight for general $n$.

% By Ramsey's theorem and the inertia bound, we know that
% $$n_{\geq 0}(G) \cdot n_{\geq 0}(\overline{G})  \geq \alpha(G) \cdot \omega(G) = \Omega(\log^2 n)$$
% but we are not sure if this is optimal.

If we restrict attention to the \emph{unweighted} inertia of a graph \( G \), we can make some progress. The additive half of the Nordhaus--Gaddum type inequalities has essentially been established in \cite[Theorem 7]{Elphick2017}, where the lower bound for the negative inertia implies
$$n_{\geq 0}(A_G) + n_{\geq 0}(A_{\overline{G}}) \leq n + 1.$$
The multiplicative half of the Nordhaus--Gaddum type inequalities appears to be more challenging, and we state it as a conjecture. Observe that for any graph $G$, either $G$ or its complement $\overline{G}$ is connected, so it suffices to consider the case where $G$ is connected.
\begin{conj}\label{conj:NordhausGaddumforinertia1}
Let \( G \) be a connected graph on \( n \) vertices, and let \( \overline{G} \) denote its complement. Then we have
\[n_{\geq 0}(A_G) \cdot n_{\geq 0}(A_{\overline{G}}) 
 \geq n. \]
\end{conj}
This multiplicative bound is tight, for example, for \( K_n \), \( P_4 \), the Beineke non-line graph G8, \texttt{Self-Complementary(9, 17)} and \texttt{Triangulated(8, 11)}. We have verified Conjecture~\ref{conj:NordhausGaddumforinertia1} numerically for all graphs with at most $9$ vertices, as well as for all graphs in the Wolfram Mathematica graph database with up to $100$ vertices. As examples, we prove that Conjecture~\ref{conj:NordhausGaddumforinertia1} holds for two important classes of graphs:

\noindent\textbf{Strongly regular graphs:} Let \( G \) be a SRG with spectrum \( (d^1, r^f, s^g) \). Then
\[
n_{\ge 0}(A_G) = 1 + f(G), \qquad n_{\ge 0}(A_{\overline{G}}) = 1 + f(\overline{G}).
\]
Using the facts that \( f(G) = g(\overline{G}) \) and \( g(G) = f(\overline{G}) \) (see~\cite{Elphick2017}), we obtain
\[
n_{\ge 0}(A_G) \cdot n_{\ge 0}(A_{\overline{G}}) = \left(1 + f(G)\right)\left(1 + g(G)\right) = n + f(G)g(G) \geq n.
\]

\noindent\textbf{Perfect graphs:} It is known that the complement of a perfect graph is also perfect~\cite[Theorem~11.23]{Lovasz19}. Therefore, with $\omega(G)$ denoting the clique number and $\alpha(G)$ the independence number of $G$, we have
\[
n_{\ge 0}(A_G) \cdot n_{\ge 0}(A_{\overline{G}}) \geq \alpha(G) \cdot \alpha(\overline{G}) = \omega(\overline{G}) \cdot \omega(G) = \chi(\overline{G}) \cdot \chi(G) \geq n,
\]
where the first inequality follows from the inertia bound~\eqref{eq:inertia-bound-1-1}, and the last inequality follows from the classical Nordhaus--Gaddum bound.

\subsection{Inertia and minrank}
We now discuss another graph parameter that appears to be related to inertia, namely the \emph{minrank}, also known as the \emph{Haemers bound}. 
\begin{defn}
Let \( G = (V, E) \) be a graph, and let \( \mathbb{F} \) be a finite or infinite field. A matrix \( M \in \mathbb{F}^{V \times V} \), not necessarily symmetric, is said to \emph{fit} \( G \) if
\begin{itemize}
    \item \( M_{vv} = 1 \) for all \( v \in V \), and
    \item \( M_{uv} = 0 \) whenever \( u \neq v \) and \( uv \notin E \).
\end{itemize} The \emph{minrank} of \( G \) over \( \mathbb{F} \), denoted by \( \mathcal{H}(G; \mathbb{F}) \), is the minimum rank of a matrix \( M \) over \( \mathbb{F} \) that fits \( G \).
\end{defn} This parameter was introduced by Haemers~\cite{Haemers1, Haemers2} to provide upper bounds on the Shannon capacity of graphs and to construct counterexamples to three conjectures of Lov\'{a}sz. It possesses several desirable properties. For example, it provides a powerful upper bound for the Shannon capacity \( \Theta(G) \)~\cite{Haemers1}:
\[
\mathcal{H}(G; \mathbb{F}) \geq \Theta(G),
\]
and it satisfies a Nordhaus--Gaddum type inequality~\cite[Section~5]{AlonBalla2020}:
\[
\mathcal{H}(G; \mathbb{F}) \cdot \mathcal{H}(\overline{G}; \mathbb{F}) \geq |G|.
\]

In a personal communication, Igor Balla observed that the Haemers bound with $\FF = \mathbb{R}$ also upper bounds the inertia.
\begin{prop}[\cite{Balla2025}]
    For any graph $G$, we have
    $$n_{\geq 0}(G) \leq 2\mathcal{H}(G; \mathbb{R}).$$
\end{prop}
\begin{proof}
Let \( M \in \mathbb{R}^{V \times V} \) be a matrix of rank \( r \) that fits \( G \). Define the symmetric matrix
\[
A := \frac{1}{2}\left(M + M^\top\right) - I.
\]
Then \( A \) is a weighted adjacency matrix of \( G \). Since \( \operatorname{rank}\left(M + M^\top\right) \leq 2r \), it follows that \( M + M^\top \) has rank at most \( 2r \). All eigenvalues of \( A \) are shifted down by 1 from those of \( \frac{1}{2}\left(M + M^\top\right) \), so \( A \) has at most \( 2r \) eigenvalues greater than \( -1 \), and thus at most \( 2r \) non-negative eigenvalues. Therefore,
\[
n_{\geq 0}(G) \leq 2r \leq 2\mathcal{H}(G; \mathbb{R}),
\] as desired. \end{proof}

% \begin{proof}
%     If a matrix $M$ over $\mathbb{R}$ fits $G$ and has rank $r$, then the matrix 
%     $$A = \frac{1}{2}(M + M^\top) - I$$
%     is a weighted adjacency matrix of $G$, and has at most $2r$ eigenvalues above $-1$.
% \end{proof}

It is natural to ask whether a reverse inequality might hold.

\begin{ques}
Does there exist a function \( f : \mathbb{N} \to \mathbb{R} \) with $\lim_{x \to \infty} f(x) = +\infty$ such that
    \[
    n_{\geq 0}(G) \geq f\bigl(\mathcal{H}(G; \mathbb{R})\bigr)
    \]
    holds for every graph \( G \)? In particular, can \( f \) be taken to be a polynomial function, i.e., \( f(x) = x^c \) for some constant \( c > 0 \)?
\end{ques}

A similar question can be asked for the fractional version of minrank, defined by Blasiak~\cite{Blasiak2013} and Bukh--Cox~\cite{BukhCox2019}.

% It seems interesting to determine if a reverse relation holds.
% \begin{conj}
%     Does there exist some function $f: \mathbb{N} \to \mathbb{N}$ such that
%     $$n_{\geq 0}(G) \geq f(\mathcal{H}(G; \mathbb{R}))$$
%     hold for any graph $G$? Can the function $f$ be taken to be $f(x) = x^c$ for some constant $c > 0$?
% \end{conj}
% The same question can be asked for the fractional version of the minrank, defined by Blasiak \cite{Blasiak2013} and Bukh--Cox \cite{BukhCox2019}.

Finally, Kwan and Widgerson~\cite[Conjecture~4.1]{Kwan2024} conjectured that if $G$ is the Erd\H{o}s--R\'enyi random graph $\mathbb{G}(n, \frac{1}{2})$, then $n_{\geq 0}(G) = \Omega\left(\frac{n}{\log n}\right)$ with high probability. We mention that the analogous problem for minrank has been solved by Alon, Balla, Gishboliner, Mond, and Mousset~\cite{AlonBalla2020}, who showed that if $G \sim \mathbb{G}(n, \frac{1}{2})$, then
$$\mathcal{H}(G; \mathbb{R}) = \Omega\left(\frac{n}{\log n}\right)$$
with high probability. It would be interesting to investigate whether their techniques can be adapted to the inertia setting.

\appendix

\section{Inertia bound via orthogonal rank}
\label{sec:orthogonal-rank}
We first recall the definitions of orthogonal rank and projective rank, as discussed in~\cite{Wocjan2019}.
\begin{defn}
    Let $G = (V, E)$ be a graph. The \emph{orthogonal rank} $\xi(G)$ is the smallest dimension $d$ such that we can assign to each vertex $i \in V$ a unit vector $\mathbf{x}_i \in \mathbb{R}^d$ with the property that $\langle \mathbf{x}_i, \mathbf{x}_j \rangle = 0$ for any edge $ij \in E$.
\end{defn}
\begin{defn}
A \emph{\( d/r \)-representation} of a graph \( G = (V, E) \) is a collection of rank-\( r \), \( d \times d \) orthogonal projectors \( \{P_v\}_{v \in V(G)} \) such that \( P_v P_w = 0 \) for all edges \( vw \in E \). The \emph{projective rank} of \( G \) is defined as
\[
\xi_f(G) = \inf_{d, r} \left\{ \frac{d}{r} : G \text{ has a } d/r \text{-representation} \right\}.
\]
\end{defn}

The following lemma is a direct corollary of~\cite[Theorem~4 and Remark~7]{Wocjan2019}, and provides a lower bound on the inertia of a graph in terms of its orthogonal rank.
\begin{lem}\label{lem:Elphickorthogonalrank1}
Let \( G \) be a simple graph. Then we have
\[
n_{\geq 0}(G) \geq \frac{|G|}{\xi(G)}.
\]
\end{lem}
We are now ready to prove Proposition~\ref{prop:large-indep}.
\begin{proof}[Proof of Proposition~\ref{prop:large-indep}]
     The second and third inequalities follow easily by substituting particular values of $k$ into the first inequality. Therefore, we only prove the first inequality.
     
     A celebrated result of Alon and Szegedy~\cite{AlonSzegedy1999} shows that for some absolute constant $c > 0$, for every positive integers $k, d$, there exists a family of at least $d^{\frac{c \log (k + 2)}{\log\log(k + 2)}}$ vectors in $\mathbb{R}^d$ such that any $(k + 1)$ members contain an orthogonal pair. We set 
     $$d = \lceil n^{\frac{\log\log(k + 2)}{c \log (k + 2)}}\rceil.$$ Then there exists a family of \( N \) unit vectors in \( \mathbb{R}^d \), with \( N \geq n \), such that any \( k + 1 \) members contain an orthogonal pair. In particular, we may select a subset of \( n \) such vectors still satisfying this property. Let \( G \) be the graph whose vertices correspond to these \( n \) vectors, with an edge between any two vectors that form an orthogonal pair. We have \( \alpha(G) \leq k \), since any \( k + 1 \) vectors form at least one edge in \( G \) by construction. Furthermore, the vectors themselves give an orthogonal representation of $G$, so the orthogonal rank satisfies $\xi(G) \leq d$. Applying Lemma~\ref{lem:Elphickorthogonalrank1}, we obtain
    $$n_{\geq 0}(G) \geq \frac{\abs{G}}{\xi(G)} \geq \frac{n}{d} \geq n^{1 - \frac{2\log\log(k + 2)}{c \log (k + 2)}},$$
    as desired. 
\end{proof}

     % Then there exists a family of $n$ vectors in $\mathbb{R}^d$ such that any $(k + 1)$ members contain an orthogonal pair. 

We proceed to prove Proposition~\ref{prop:projective-independence-2}.
\begin{proof}[Proof of Proposition~\ref{prop:projective-independence-2}]
    Let $\{P_v\}_{v \in V}$ be a $d/r$-representation of $G$. Consider the matrix
    $$X = \sum_{v \in V} P_v.$$
    Then 
    $$X^3 = \sum_{u,v,w \in V} P_u P_v P_w.$$
    We now compute the trace of $X^3$. When $u,v,w$ are pairwise distinct, at least one of $uv, vw, wu$ is an edge. In each case we have
    $$\tr(P_u P_v P_w) = \tr( P_v P_w P_u) = 0.$$
    So it suffices to consider the summand where at least two of $u, v, w$ are equal. For every vertex $u$, $P_u$ is an orthogonal projector, so $P_u^2 = P_u$. Thus we can simplify
    $$\tr(X^3) = 3 \sum_{u,v \in V} \tr(P_u P_v) - 2 \sum_{u \in V} \tr(P_u) = 3 \tr(X^2) - 2 \tr(X).$$
    Let \( \lambda(X) \) denote the multiset of eigenvalues of \( X \). Then we have
\[
\sum_{\lambda \in \lambda(X)} (\lambda^3 - 3\lambda^2 + 2\lambda) = 0.
\]
    Since $X$ is a sum of orthogonal projections, it is positive semidefinite, and hence all eigenvalues of \( X \) are non-negative. We can verify directly that for every $\lambda \geq 0$, we have
    $$\lambda^3 - 3 \lambda^2 + 2 \lambda \geq 2\lambda - 4.$$
    So we obtain
    $$\sum_{\lambda \in \lambda(X)} (2 \lambda - 4) \leq 0.$$
    This shows that $\tr(X) \leq 2 d$. As we also have $\tr(X) = rn$, we obtain
    $$\frac{d}{r} \geq \frac{n}{2},$$
    as desired.
\end{proof}
% \begin{remark}
% An identical proof shows that the projective packing number $\Tilde{\alpha}(G)$, defined in Roberson's thesis \cite{Roberson}, is also equal to $2$ when $\alpha(G) = 2$.\end{remark}

\begin{remark}
An identical argument shows that the \emph{projective packing number} $\tilde{\alpha}(G)$, as defined in Roberson's thesis~\cite{Roberson}, is also equal to $2$ whenever $\alpha(G) = 2$. It is also worth noting that \cite{Roberson} observed the inequality \(\tilde{\alpha}(G) \cdot \xi_f(G) \geq |G|\), with equality if $G$ is vertex-transitive.
\end{remark}

\end{document}